\documentclass[11pt]{amsart}
\usepackage{amsmath,amsfonts,amssymb,amsthm,amscd}
\usepackage{graphicx}
\usepackage{mathrsfs}
\usepackage{upgreek}
\usepackage[matrix,arrow,curve]{xy}
\begin{document}

\newcommand\hide[1]{\commented{gray}{Hidden:}{#1}}
\renewcommand\hide[1]\empty

\theoremstyle{plain}
\newtheorem{thm}{Theorem}
\newtheorem{lmm}{Lemma}
\newtheorem{claim}{Claim}
\newtheorem{prop}{Proposition}
\newtheorem{coro}{Corollary}
\newtheorem*{tm}{Theorem}
\newtheorem*{lm}{Lemma}
\newtheorem*{clm}{Claim}
\newtheorem*{prp}{Proposition}
\newtheorem*{cor}{Corollary}

\theoremstyle{definition}
\newtheorem{dfn}{Definition}
\newtheorem{prbl}{Problem}
\newtheorem{conj}{Conjecture}
\newtheorem*{df}{Definition}
\newtheorem*{nt}{Notation}
\newtheorem*{obs}{Observation}
\newtheorem*{exm}{Example}
\newtheorem*{exms}{Examples}
\newtheorem*{prb}{Problem}
\newtheorem*{cnj}{Conjecture}

\theoremstyle{remark}
\newtheorem*{pf}{Proof}
\newtheorem*{rmk}{Remark}
\newtheorem*{q}{Question}
\newtheorem*{qs}{Questions}
\newtheorem*{ex}{Example}
\newtheorem*{exs}{Examples}
\newtheorem*{ack}{Acknowledgements}
\newtheorem*{fact}{Fact}
\newtheorem*{facts}{Facts}

\renewcommand{\varPhi}{\mathbf{\Phi}}
\renewcommand{\varPsi}{\mathbf{\Psi}}

\newcommand{\ob}{[}
\newcommand{\cb}{]}

\newcommand{\dom}{ {\mathop{\mathrm {dom\,}}\nolimits} }
\newcommand{\ran}{ {\mathop{\mathrm{ran\,}}\nolimits} }
\newcommand{\Ran}{ {\mathop{\mathrm{Range\,}}\nolimits} }

\newcommand{\arity}{ {\mathop{\mathrm {ar}}\nolimits} }
\newcommand{\cf}{ {\mathop{\mathrm {cf\,}}\nolimits} }
\newcommand{\inter}{ {\mathop{\mathrm {int\,}}\nolimits} }
\newcommand{\cl}{ {\mathop{\mathrm {cl\,}}\nolimits} }
\newcommand{\cll}{ {\mathop{\mathrm {cl^{-}\,}}\nolimits} }
\newcommand{\clu}{ {\mathop{\mathrm {cl^{+}\,}}\nolimits} }
\newcommand{\lcl}{ {\mathop{\mathrm {lcl\,}}\nolimits} }
\newcommand{\rcl}{ {\mathop{\mathrm {rcl\,}}\nolimits} }
\newcommand{\cof}{ {\mathop{\mathrm {cof\,}}\nolimits} }
\newcommand{\add}{ {\mathop{\mathrm {add\,}}\nolimits} }
\newcommand{\sat}{ {\mathop{\mathrm {sat\,}}\nolimits} }
\newcommand{\tc}{ {\mathop{\mathrm {tc\,}}\nolimits} }
\newcommand{\unif}{ {\mathop{\mathrm {unif\,}}\nolimits} }

\newcommand{\alg}{\mathop{\mathrm {alg\,}}\nolimits} 

\newcommand{\uhr}{\upharpoonright}
\newcommand{\lra}{ {\leftrightarrow} }
\newcommand{\ot}{ {\mathop{\mathrm {ot\,}}\nolimits} }
\newcommand{\pr}{ {\mathop{\mathrm {pr}}\nolimits} }
\newcommand{\cnc}{ {^\frown} }
\newcommand{\image}{\/``\,}
\newcommand{\scc}{\beta\!\!\!\!\beta}
\newcommand{\filt}{\varepsilon\!\!\!\!\!\;\varepsilon}
\newcommand{\pws}{P\!\!\!\!\!P}
\newcommand{\ol}{\overline}
\newcommand{\wh}{\widehat}
\newcommand{\wt}{\widetilde}
\newcommand{\btu}{\bigtriangleup}
\newcommand{\mbf}{\mathbf}

\newcommand{\supp}{\mathrm{supp}}
\newcommand{\defin}{\mathrm{def}}
\newcommand{\Pos}{\mathrm{Pos}}
\newcommand{\id}{\mathrm{id}}
\newcommand{\diam}{\mathrm{diam}}

\newcommand{\CK}{\mathrm{CK}}
\newcommand{\low}{\mathrm{low}}

\newcommand{\DC}{\mathrm{DC}}
\newcommand{\AC}{\mathrm{AC}}
\newcommand{\PD}{\mathrm{PD}}
\newcommand{\AD}{\mathrm{AD}}
\newcommand{\Det}{\mathrm{Det}}

\newcommand{\ZFC}{\mathrm{ZFC}}

\newcommand{\Ha}{\mathrm{H}}
\newcommand{\W}{\mathrm{W}}

\newcommand{\Sc}{\mathrm{Sc}}
\newcommand{\Ap}{\mathrm{Ap}}

\pagestyle{headings}

\title[Reduction of projective sets in 
Tychonoff spaces]{On reduction and separation 
of projective sets in Tychonoff spaces}

\author{Denis I.~Saveliev}
\date{} 
        %

\hide{
\thanks{
\noindent
This work was partially supported by grant 17-01-00705 
of Russian Foundation for Basic Research and carried 
out at Institute for Information Transmission Problems 
of the Russian Academy of Sciences.
}
}

\thanks{
\!\!{\em MSC~2010\/}:
Primary 
03E15, 
03E60, 
54H05, 
Secondary 
54C10, 
54D30. 
}


\begin{abstract}
We show that for every Tychonoff space~$X$ 
and Hausdorff operation~$\mathbf\Phi$, the class 
$\mathbf\Phi(\mathscr Z,X)$ generated from zero 
sets in $X$ by~$\mathbf\Phi$ has the reduction 
or separation property if the corresponding class 
$\mathbf\Phi(\mathscr F,\mathbb R)$ of sets 
of reals has the same property.

In particular, under Projective Determinacy, these 
properties of such projective sets in $X$ form  
the same pattern as the First Periodicity Theorem 
states for projective sets of reals: the classes
$\mathbf\Sigma^{1}_{2n}(\mathscr Z,X)$ and 
$\mathbf\Pi^{1}_{2n+1}(\mathscr Z,X)$ have  
the reduction property while the classes
$\mathbf\Pi^{1}_{2n}(\mathscr Z,X)$ and 
$\mathbf\Sigma^{1}_{2n+1}(\mathscr Z,X)$ 
have the separation property. 
\end{abstract}

\maketitle


\section{Introduction}

In the sequel, 
$\mathscr F$~denotes the class of closed sets,
$\mathscr G$~of open sets,
$\mathscr K$~of compact sets,
$\mathscr Z$~of zero sets, i.e., pre-images of 
the point~$0$ of the closed segment~$[0,1]$ of 
the real line under continuous maps; finally,
$\mathscr S$~denotes an unspecified class of subsets.
These classes are treated as operators applied to 
a~given topological space~$X$ so $\mathscr F(X)$ 
consists of all closed sets in~$X$, etc. 
For an arbitrary class~$\mathscr S$ we let
$\mathscr S(X)=\mathscr S\cap\mathscr P(X)$. Let also 
$\mathscr S(Y)\uhr X=\{S\cap X:S\in\mathscr S(Y)\}$. 
As well-known,
$\mathscr K\subseteq\mathscr F$ for Hausdorff spaces,
and so $\mathscr K=\mathscr F$ for Hausdorff compact 
(and hence normal) spaces; 
$\mathscr Z\subseteq\mathscr F\cap\mathscr G_\delta$,
moreover, $\mathscr Z=\mathscr F\cap\mathscr G_\delta$ 
for normal spaces (see, e.g., \cite{Engelking}, 
1.5.11; this fails for Tychonoff spaces), and so 
$\mathscr Z=\mathscr F$ for perfectly normal spaces.

\hide{
\vskip+1em
\footnotesize 

In ZFC, $\mathscr S$ is a~formula with 
one or more parameters.
\vskip+1em

Does 
$
\mathscr K\cap\mathscr G_\delta
=\mathscr K\cap\mathscr Z
$
hold for all (or at least all Tychonoff) spaces?

Recall a~topological space $X$ is completely regular 
iff $\mathscr Z(X)$~is a~closed basis for its topology, 
and $X$~is Tychonoff iff it can be embedded in 
a~Tychonoff cube. The cube $[0,1]^\kappa$ is universal 
for Tychonoff spaces of weight~$\kappa$ (see, e.g., 
\cite{Engelking}, 2.3.23).

A~topological space~$X$ is 
a~$\mathscr G_\delta$-{\em space} iff 
$\mathscr F(X)\subseteq\mathscr G_\delta(X)$
(see \cite{Steen Seebach}, p.~162). 
A~{\em perfectly normal} space is a~normal 
$\mathscr G_\delta$-space. $\mathbb R$~endowed 
with the K-topology is an example of 
a~$\mathscr G_\delta$-space that is not normal.
The Sorgenfrey line~$S$ is an example of 
a~perfectly normal space that is not metrizable.
The square $S\times S$ is an example of 
a~Tychonoff space~$X$ with 
$
\mathscr Z(X)\subset
(\mathscr F\cap\mathscr G_\delta)(X)
$. 

\vskip+1em
\footnotesize

An operation on classes of sets is 
\emph{set-theoretic} iff\ldots\quad
A~set-theoretic operation is 
\emph{monotone} iff\ldots\quad

Note that all results and proofs below remain true 
for $\delta\/s$-operations of arbitrary arities
(and even, mutatis mutandis, for any set-theoretical 
operations).
\vskip+1em
\normalsize
}

Albeit a~part of results of this paper 
remains true for arbitrary Hausdorff operations, 
below we consider the operations only of 
countable arity, which we define as follows. 
A~{\em Hausdorff operation} 
(or $\delta s$-\emph{operation})~$\mathbf{\Phi}$ 
applied to a~family $(A_n)_{n<\omega}$ 
of sets~$A_n$ has the form 
$$
\varPhi(A_n)_{n<\omega}=
\bigcup_{f\in S}\bigcap_{n\in\omega}A_{f\uhr n}
$$
for some $S\subseteq\omega^\omega$, called the 
\emph{base} of~$\varPhi$, where sets $A_{f\uhr n}$ 
are identified with the sets~$A_n$ under a~fixed 
bijection of $\omega$ onto $\omega^{<\omega}$. 
E.g., in the simplest case $S=\omega^\omega$ 
we get Alexandroff's A-operation. 
Clearly, the operations given by bases 
$S$ and $\widetilde{S}$ coincide, where 
$
\widetilde{S}=\{g\in\omega^\omega:
(\exists f\in S)\,\ran f=\ran g\} 
$
is the {\em completion} of~$S$.
A~$\varPhi$-{\em set} is a~set obtained by~$\varPhi$. 
We let $\varPhi(\mathscr S,X)$ to denote 
the class of $\varPhi$-sets generated by 
sets in $\mathscr S(X)$, i.e., 
$$
\mathbf{\Phi}(\mathscr S,X)=
\{\varPhi(A_n)_{n<\omega}:
(A_n)_{n<\omega}\in\mathscr S(X)^\omega\}.
$$
By $\varPhi(\mathscr S)$ we mean the union 
of $\varPhi(\mathscr S,X)$ for all~$X$.

These operations were introduced in late 1920s 
by Hausdorff and independently by Kolmogorov, 
who established a~series of important results; 
later on the theory was generalized to operations 
of arbitrary arity. Early basic results can be 
found in~\cite{Kantorovich Livenson}.
We also refer the reader to two reviews, 
\cite{Kanovei 1988} and \cite{Choban 1989},
the first of which stresses on Kolmogorov's 
R-operation (on operations) and its connection 
to game quantifiers, the second on using in 
general topology, and to the literature there. 
Some newer interesting results can be found in 
\cite{Dasgupta} and~\cite{Dougherty}.
Finally, $\kappa$-{\em Suslin sets}, which 
are $\varPhi$-sets for $\varPhi$ with the 
base~$\kappa^\omega$, play a~crucial role in 
modern descriptive set theory, see 
\cite{Kechris et al 2008} and 
\cite{Moschovakis}, Chapter~2.

\hide{
Some newer interesting results can be found in
papers \cite{Dougherty} and~\cite{Dasgupta};
the first characterizes classes that have the form
$\varPhi(\mathscr F\cap\mathscr G,\omega^\omega)$ 
for some $\varPhi$ of arbitrary arity, the second
describes families $\mathscr S(X)$ with an uncountable 
Polish~$X$ such that the class of Borel sets in $X$ 
has the form $\varPhi(\mathscr S,X)$ 
for some $\omega$-ary~$\varPhi$.
}


Below we apply results on $\varPhi$-sets 
to the Borel and projective hierarchies.

The {\em Borel hierarchy} generated by sets in 
$\mathscr S(X)$ is defined in the standard way by 
alternating countable unions and complements; 
as usual, 
$\mathbf\Sigma^{0}_{\alpha}(\mathscr S,X)$, 
$\mathbf\Pi^{0}_{\alpha}(\mathscr S,X)$, and 
$\mathbf\Delta^{0}_{\alpha}(\mathscr S,X)$ denote 
the $\alpha$th additive, multiplicative, and 
self-dual classes of the resulting hierarchy. 
So, e.g., $\mathbf\Sigma^{0}_{2}(\mathscr F,X)$ 
is $\mathscr F_\sigma(X)$ and 
$\mathbf\Pi^{0}_{2}(\mathscr F,X)$ 
is $\mathscr G_\delta(X)$. 
As complements of $\varPhi$-sets are $(-\varPhi)$-sets, 
where $-\varPhi$ is the operation dual to~$\varPhi$, 
it is easy to prove by induction on~$\alpha$ that each 
of Borel classes is of form $\varPhi(\mathscr S,X)$
for an appropriate~$\varPhi$.

The {\em projective hierarchy} generated by sets 
in $\mathscr S(X)$ for a~Polish space~$X$ is 
defined by alternating projections of subsets of 
$X\times\omega^\omega$ and complements; as usual, 
$\mathbf\Sigma^{1}_{n}(\mathscr S,X)$, 
$\mathbf\Pi^{1}_{n}(\mathscr S,X)$, and 
$\mathbf\Delta^{1}_{n}(\mathscr S,X)$ denote 
the $n$th additive, multiplicative, and self-dual 
classes of the resulting projective hierarchy. So, 
e.g., $\mathbf\Sigma^{1}_{1}(\mathscr F,\mathbb R)$ 
and $\mathbf\Pi^{1}_{1}(\mathscr F,\mathbb R)$ consist 
of A-sets and CA-sets of reals, respectively. 
Several alternative definitions lead to the same 
projective classes in Polish case; this is no 
longer the case, however, if $X$~is not Polish
(see~\cite{Miller}). A~definition 
suitable for our purposes is via Hausdorff operations. 
Before giving this, we recall the Fundamental Theorem 
on Projections by Kantorovich and Livenson 
(\cite{Kantorovich Livenson}, p.~264), 
which states that whenever $X$ is Polish (or even 
Suslin, i.e., a~continuous image of $\omega^\omega$) 
then the class of projections onto $X$ of sets 
in $\varPhi(\mathscr F,X\times\omega^\omega)$ 
is itself of form $\varPsi(\mathscr F,X)$ for 
the Hausdorff operation~$\varPsi$ given by a~base 
$S\in\varPhi(\mathscr F_\sigma,\omega^\omega)$. 
It easily follows by induction on~$n$ that each 
of projective classes in Polish spaces is of form 
$\varPhi(\mathscr F,X)$ for an appropriate~$\varPhi$. 
Now we define projective classes in arbitrary spaces as 
the classes of $\varPhi$-sets for $\varPhi$ such that 
the corresponding projective class in the Polish case 
consists of $\varPhi$-sets. This approach is easily 
continued to $\sigma$-projective sets as they are 
defined in \cite{Kechris} (see also
\cite{Di Prisco et al 1982},~\cite{Di Prisco et al 1987}).


We use also the following notation:
$-\mathscr S(X)=\{X\setminus S:S\in\mathscr S(X)\}$ 
and 
$
\mathbf{\Delta}(\mathscr S,X)=
\mathscr S(X)\cap-\mathscr S(X).
$
Given a~map $F:X\to Y$ and sets $A\subseteq X$ 
and $B\subseteq Y$, let $FA=\{F(x):x\in A\}$ 
and $F^{-1}B=\{x:\exists y\in B\;F(x)=y\}.$ 
Moreover, $F\mathscr S=\{FS:S\in\mathscr S\}$ 
and $F^{-1}\mathscr S=\{F^{-1}S:S\in\mathscr S\}$.

Sets $A,B$ are {\em reduced} by sets $C,D$ iff 
$C\subseteq A$, $D\subseteq B$,
$C\cap D=\emptyset$, and $C\cup D=A\cup B$; and 
{\em separated} by a~set~$C$ iff $A\subseteq C$ 
and $B\cap C=\emptyset$ (in the latter case $A,B$ 
are assumed to be disjoint). Two following properties 
of classes of sets are the main subject of this note:
$\mathscr S(X)$~has 
\begin{itemize}
\item[(i)] 
the {\it reduction\/} property iff 
every $A,B\in\mathscr S(X)$ are 
reduced by some $C,D\in\mathscr S(X)$, 
\item[(ii)] 
the {\it separation\/} property iff 
every disjoint $A,B\in\mathscr S(X)$ are 
separated by some $C\in\mathbf{\Delta}(\mathscr S,X)$.
\end{itemize}

These properties were formulated and established 
for lower projective classes by Lusin, Kuratowski, 
and Novikov: reduction for 
$\mathbf\Pi^{1}_1(\mathscr F,\mathbb R)$ and 
$\mathbf\Sigma^{1}_2(\mathscr F,\mathbb R)$, 
and separation for dual 
$\mathbf\Sigma^{1}_1(\mathscr F,\mathbb R)$ 
and $\mathbf\Pi^{1}_2(\mathscr F,\mathbb R)$;
earlier Sierpi{\'n}ski and Lavrentieff independently  
proved separation for Borel classes 
$\mathbf\Pi^{0}_\alpha(\mathscr F,\mathbb R)$;
for arbitrary projective classes the properties 
were formulated by Addison who also proved that 
$V=L$ implies reduction 
in $\mathbf\Sigma^{1}_2(\mathscr F,\mathbb R)$ for 
all $n\ge2$. In subsequent studies, the stronger 
properties related to the existence of {\em norms} 
and {\em scales} were isolated and established 
for all projective classes under PD, the Projective 
Determinacy, where they form a~pattern of period~$2$, 
the fundamental facts known as the First and Second 
Periodicity Theorems, proved by Martin and Moschovakis. 
We do not formulate here neither PD 
(a~weak form of~AD, the Axiom of Determinacy, proposed 
by Mycielski and Steinhaus whose relative consistency 
was proved by Martin, Steel, and Woodin) 
nor these stronger properties and refer the reader 
to \cite{Kanamori}, \cite{Kechris}, 
\cite{Kechris et al 2008}, \cite{Kechris et al 2011}, 
\cite{Moschovakis} for a~further (including historical) 
information.

The plan of this paper is roughly as follows. First 
we study when classes of $\varPhi$-sets are preserved 
under subspaces and certain maps in the pre-image and 
image directions. Then we combine obtained results to 
transfer the reduction and separation properties of 
these classes in the pre-image direction. This leads 
us to the main result (Theorem~\ref{t: red tych}) 
stating that for every Tychonoff space~$X$ and 
Hausdorff operation~$\mathbf\Phi$, the class 
$\mathbf\Phi(\mathscr Z,X)$ has the reduction or 
separation property if the corresponding class 
$\mathbf\Phi(\mathscr F,\mathbb R)$ of sets 
of reals has the same property. It follows
(Corollary~\ref{c: period tych}) that under PD,  
these properties of such projective sets in $X$ 
form  the same pattern as the First Periodicity 
Theorem states for projective sets of reals.

The proofs of this paper are straightforward, and 
its main meaning lies in the observation (which may 
seem somewhat surprising) that certain descriptive 
properties of Polish spaces are transferred to spaces 
very far from Polish ones. In particular, this shows 
that AD in a~weak form consistent with~$\ZFC$, which 
relates immediately to these properties of the real line, 
have an unexpected impact to general topological spaces.

\section{Auxiliary results}

We start with a~few easy observations concerning  
an interplay between reduction, separation, and 
restrictions of classes of $\varPhi$-sets in 
a~given space to its subspaces.

\begin{lmm}\label{l: red in subspace}
Let $\mathscr S$ be a~class and $X,Y$ some sets. 
\begin{itemize}
\item[(i)] 
If $\mathscr S(X)$ has reduction then 
$-\mathscr S(X)$ has separation.
\item[(ii)] 
If $\mathscr S(Y)$ has reduction (separation)
then $\mathscr S(Y)\uhr X$ has the same property.
\end{itemize}
\end{lmm}

\begin{proof}
Clear.
\end{proof}


\hide{
\section*{Subspaces vs Hausdorff operations}
}

\begin{lmm}\label{l: hausd distr}
Let $\varPhi$ be a~Hausdorff operation. Then:
\begin{itemize}
\item[(i)] 
finite intersections and unions 
distribute over~$\varPhi$,
\item[(ii)] 
$
\varPhi(\mathscr S(Y)\uhr X)=
\varPhi(\mathscr S,Y)\uhr X
$
for any $\mathscr S$ and $X,Y$.
\end{itemize}
\end{lmm}

\begin{proof}
(i). 
If $S\subseteq\omega^\omega$ is a~base of~$\varPhi$,
we have
\begin{align*}
\varPhi(B_n)_{n\in\omega}\cap X=
(\bigcup_{f\in S}\bigcap_{n\in\omega}B_{f\uhr n})\cap X=
\bigcup_{f\in S}\bigcap_{n\in\omega}(B_{f\uhr n}\cap X)=
\varPhi(B_n\cap X)_{n\in\omega}.
\end{align*}
Twice applying this, we see that binary intersections 
distribute over~$\varPhi$:
$
\varPhi(A_m)_{m\in\omega}\cap\varPhi(B_n)_{n\in\omega}=
\varPhi(\varPhi(A_m\cap B_n)_{m\in\omega})_{n\in\omega},
$
and similarly for finite intersections.
The case of unions is analogous.

(ii).
Immediate from the considered particular case of~(i).
\end{proof}

\begin{lmm}\label{l: hausd in subspace}
For any Hausdorff operation~$\varPhi$, 
class~$\mathscr S$, and sets $X\subseteq Y$,
\begin{itemize}
\item[(i)] 
if 
$\mathscr S(X)\subseteq\mathscr S(Y)\uhr X$ then 
$\varPhi(\mathscr S,X)\subseteq\varPhi(\mathscr S,Y)\uhr X$, 
\item[(ii)] 
if
$\mathscr S(Y)\uhr X\subseteq\mathscr S(X)$ then 
$\varPhi(\mathscr S,Y)\uhr X\subseteq\varPhi(\mathscr S,X)$. 
\end{itemize}
\end{lmm}

\begin{proof} 
As $\varPhi$~is monotone, i.e., 
$\mathscr S\subseteq\mathscr T$ implies
$\varPhi(\mathscr S)\subseteq\varPhi(\mathscr T)$, 
this follows from Lemma~\ref{l: hausd distr}(ii).
\end{proof}

\begin{coro}\label{c: hausd in subspace}
Let $\varPhi$~be a~Hausdorff operation, and let
$\mathscr S$ and $X\subseteq Y$ be such that 
$\mathscr S(X)=\mathscr S(Y)\uhr X.$ Then:
\begin{itemize}
\item[(i)] 
$\varPhi(\mathscr S,X)=\varPhi(\mathscr S,Y)\uhr X$,
\item[(ii)] 
if $\varPhi(\mathscr S,Y)$ has reduction 
(separation) then $\varPhi(\mathscr S,X)$
has the same property.
\end{itemize}
\end{coro}

\begin{proof}
(i). 
Lemma~\ref{l: hausd in subspace}.

(ii). 
Follows from (i) and 
Lemma \ref{l: red in subspace}(ii).
\end{proof}

The assumption of Corollary~\ref{c: hausd in subspace} 
holds, e.g., if $\mathscr S$~is $\mathscr F$ or 
$\mathscr G$ for arbitrary spaces $X,Y$, and also 
if $\mathscr S$~is $\mathscr Z$ for Tychonoff spaces; 
the latter fact is established in 
Lemma~\ref{l: zero tych} below.

Note that for $\mathscr S(Y)$ closed under
finite intersections, if $X\in\mathscr S(Y)$, then 
$\mathscr S(Y)\uhr X\subseteq\mathscr S(Y)$, and so 
the assumption gives $\mathscr S(X)\subseteq\mathscr S(Y)$.
However, in Theorem~\ref{t: red tych} where 
Corollary~\ref{c: hausd in subspace} will be used, 
Tychonoff spaces~$X$ will be considered as arbitrary 
subspaces of $Y=[0,1]^\kappa$ without a~guarantee 
of being a~member of~$\mathscr S(Y)$.

\hide{
\vskip+1em
\footnotesize

A~consequence on non-degenerate classes:
Let $\varPhi$ and $\varPsi$~be Hausdorff operations,
and $\mathscr S(X)=\mathscr S(Y)\uhr X.$ Then 
$
\varPhi(\mathscr S,X)\setminus
\varPsi(\mathscr S,X)=
\varPhi(\mathscr S,Y)\uhr X\setminus
\varPsi(\mathscr S,Y)\uhr X.
$

\vskip+1em
\normalsize

%
}


\hide{
\section*{Pre-images vs Hausdorff operations}
}

Two results below, Corollaries 
\ref{c: hausd vs preimage} 
and~\ref{c: hausd vs image}, provide 
conditions under which classes of $\varPhi$-sets 
are preserved under maps in the image and pre-image 
direction, respectively, for arbitrary~$\varPhi$.

\begin{lmm}\label{l: hausd vs preimage}
Let $\varPhi$~be a~Hausdorff operation. 
For any $X,Y$, $F:X\to Y$, and 
$(B_n)_{n\in\omega}$ in $\mathscr P(Y)$, we have
$
F^{-1}\varPhi(B_n)_{n\in\omega}=
\varPhi(F^{-1}B_n)_{n\in\omega}.
$
\end{lmm}

\begin{proof}
Let $S\subseteq\omega^\omega$ be a~base of~$\varPhi$.
Since pre-images distribute over arbitrary unions 
and intersections, we have
\begin{align*}
F^{-1}\varPhi(B_n)_{n\in\omega}=
F^{-1}\bigcup_{f\in S}\bigcap_{n\in\omega}B_{f\uhr n}=
\bigcup_{f\in S}\bigcap_{n\in\omega}F^{-1}B_{f\uhr n}=
\varPhi(F^{-1}B_n)_{n\in\omega},
\end{align*}
as required. 
\end{proof}

Given $\mathscr S$ and $F:X\to Y$, we say that 
$F$~{\em preserves $\mathscr S$} iff 
$A\in\mathscr S(X)$ implies $FA\in\mathscr S(Y)$, 
and $F^{-1}$~{\em preserves $\mathscr S$} iff 
$B\in\mathscr S(Y)$ implies $F^{-1}B\in\mathscr S(X)$. 
As usual, $F$~is {\em closed} iff it preserves $\mathscr F$, 
{\em open} iff it preserves $\mathscr G$, {\em continuous} 
iff $F^{-1}$ preserves $\mathscr F$ (or~$\mathscr G$), 
and {\em proper} iff $F^{-1}$ preserves~$\mathscr K$.
\hide{
Moreover, $F$~is {\em perfect} iff it is continuous, 
closed, and proper. 
\begin{footnotesize}
In \cite{Engelking},  $F$~is {\em perfect} iff it is 
continuous, closed, and compact-to-one; but Theorem~3.7.2 
there states that then $F$~is proper too.
\end{footnotesize}
}

\begin{coro}\label{c: hausd vs preimage}
Let $\varPhi$~be a~Hausdorff operation and $F:X\to Y$.
If $F^{-1}$ preserves $\mathscr S$, then  
$F^{-1}$ preserves $\varPhi(\mathscr S)$, i.e.,
$F^{-1}\varPhi(\mathscr S,Y)\subseteq\varPhi(\mathscr S,X)$.
\end{coro}

\begin{proof}
Lemma~\ref{l: hausd vs preimage}.
\end{proof}

E.g., if $F$~is continuous then $F^{-1}$ preserves 
each of $\varPhi(\mathscr F)$, $\varPhi(\mathscr G)$, 
$\varPhi(\mathscr Z)$, and if $F$~is proper then 
$F^{-1}$ preserves~$\varPhi(\mathscr K)$.


\hide{
\section*{Algebra of pre-images}
}

The purpose of the next series of statements is to 
construct special maps with prescribed sets as pre-images. 
Given a~map $F:X\to Y$, we consider its {\em kernel}
$\ker F=\{F^{-1}\{y\}:y\in Y\}$ and {\em algebra of 
pre-images} $\alg F=\{F^{-1}B:B\subseteq Y\}.$

\begin{lmm}\label{l: algebra of preimages}
For any $F:X\to Y$ we have
$$
\alg F=
\{A\subseteq X:F^{-1}FA=A\}.
$$
Moreover, $\alg F$ is a~complete subalgebra 
of $\mathscr P(X)$ generated by $\ker F$ 
and thus isomorphic to $\mathscr P(\ker F)$. 
Consequently, $\alg F$ is closed under 
Hausdorff operations.
\hide{
\qquad
\begin{footnotesize}
In fact, $\alg F$ is closed under all 
set-theoretic operations of any arity.
\end{footnotesize}
}
\end{lmm}

\begin{proof}
Clear.
\end{proof}


Given maps $F_i:X\to Y_i$, $i\in I$, 
their {\it diagonal product\/} is the map 
$\btu_{i\in I}F_i:X\to\prod_{i\in I}Y_i$
defined by letting for all $x\in X$,
$$
\btu_{i\in I}F_i(x)=
(F_i(x))_{i\in I}.
$$ 
The diagonal product of continuous maps~$F_i$ is 
continuous (w.r.t.~the standard product topology on 
$\prod_{i\in I}Y_i$), and moreover, it is perfect 
whenever so is at least one of them, say, $F_j$, 
and the spaces $Y_i$ for all $i\ne j$ are Hausdorff 
(see \cite{Engelking}, Theorem~3.7.9).

\begin{lmm}\label{l: algebra vs diag prod}
If $A\in\alg F_j$ for some $j\in I$, 
then $A\in\alg(\btu_{i\in I}F_i)$.
\end{lmm}

\begin{proof} 
Let $F=\btu_{i\in I}F_i$. 
If $A=F^{-1}_jB$ for some $B\subseteq Y_j$ then 
$A=F^{-1}(B\times\prod_{i\in I\setminus\{j\}}Y_i)$. 
\end{proof}


As usual, a~class~$\mathscr Y$ of topological 
spaces is {\em closed under $\kappa$~products} iff 
$(Y_\alpha)_{\alpha<\kappa}\in\mathscr Y^\kappa$ 
implies $\prod_{\alpha<\kappa}Y_\alpha\in\mathscr Y$.
E.g., the class of Polish spaces is closed under 
$\omega$~products, the class of spaces of density 
$\lambda\ge\omega$ is closed under $2^\lambda$~products 
(see \cite{Engelking}, 2.3.15), and $\mathscr K$~is 
closed under arbitrary products. Similarly, 
a~class~$\mathscr M$ of maps is {\em closed under 
$\kappa$~diagonal products} iff 
$(F_\alpha)_{\alpha<\kappa}\in\mathscr M^\kappa$ 
implies $\btu_{\alpha<\kappa}F_\alpha\in\mathscr M$. 
E.g., the classes of continuous and of perfect maps
are closed under arbitrary products.


\begin{prop}\label{p: algebra vs diag prod}
Let $\mathscr Y$ be closed under $\kappa$~products, 
$\mathscr M$ a~class of maps closed under 
$\kappa$~diagonal products, and let $\mathscr S$ be
such that for any $S\in\mathscr S(X)$ there exist 
$Y\in\mathscr Y$ and $F\in\mathscr M\cap Y^X$ 
such that $S\in\alg F$. Then for any 
$(S_\alpha)_{\alpha<\kappa}\in\mathscr S(X)^\kappa$
there exist $Y\in\mathscr Y$ and $F\in\mathscr M\cap Y^X$ 
such that $S_\alpha\in\alg F$ for all $\alpha<\kappa$. 
Consequently, $\varPhi(S_\alpha)_{\alpha<\kappa}\in\alg F$
for any (even $\kappa$-ary) Hausdorff operation~$\varPhi$. 
\end{prop}

\begin{proof} 
For each $\alpha<\kappa$ pick $Y_\alpha\in\mathscr Y$ 
and $F_\alpha\in\mathscr M\cap{Y_\alpha}^X$ with 
$S_\alpha\in\alg F_\alpha$. 
Let $Y=\prod_{\alpha<\kappa}Y_\alpha$ and 
$F=\btu_{\alpha<\kappa}F_\alpha$. Then 
$Y\in\mathscr Y$ since $\mathscr Y$ 
is closed under $\kappa$~products, 
$F\in\mathscr M\cap Y^X$ since $\mathscr M$ 
is closed under $\kappa$~diagonal products,
moreover,
$S_\alpha\in\alg F$ for all $\alpha<\kappa$
by Lemma~\ref{l: algebra vs diag prod}, 
and so $\varPhi(S_\alpha)_{\alpha<\kappa}\in\alg F$ 
by Lemma~\ref{l: algebra of preimages}. 
\hide{
\qquad
\begin{footnotesize}
this uses $\AC_\kappa$.
\end{footnotesize}
} 
\end{proof}

The following 
Proposition~\ref{p: zero sets vs diag prod} 
is essentially a~variant of 
Proposition~\ref{p: algebra vs diag prod}
where we have $\mathscr S=\mathscr Z$, $\kappa=\omega$, 
$\mathscr Y=\{[0,1]^\omega\}$, and $\mathscr M$~consists
of continuous maps witnessing that sets $A_n$ are 
in $\mathscr Z(X)$.

\begin{prop}\label{p: zero sets vs diag prod}
Let $X$~be a~topological space and 
$(A_n)_{n<\omega}\in\mathscr Z(X)^\omega$. 
Then there exists a~continuous map $F:X\to[0,1]^\omega$ 
such that $A_n\in\alg F$ for all $n<\omega$. 
Consequently, $\varPhi(A_n)_{n<\omega}\in\alg F$
for any Hausdorff operation~$\varPhi$. 
\end{prop}

\begin{proof} 
For each $n<\omega$ pick a~continuous 
$F_n:X\to[0,1]$ with $A_n=F^{-1}_n\{0\}$ 
(which is possible since $A_n$ is in~$\mathscr Z(X)$),
and thus $A_n\in\alg F_n$. Then 
$F=\btu_{n<\omega}F_n:X\to[0,1]^\omega$ is continuous, 
$A_n\in\alg F$ by Lemma~\ref{l: algebra vs diag prod}, 
and so $\varPhi(A_n)_{n<\omega}\in\alg F$ 
by Lemma~\ref{l: algebra of preimages}. 
\hide{
\qquad
\begin{footnotesize}
this uses $\AC_\omega$.
\end{footnotesize}
} 
\end{proof}

\hide{
\vskip+1em
\footnotesize

[Can we replace $[0,1]^\omega$ by $[0,1]$? 
[at least, when $X$~is compact so the sets
$F(A_n)\subseteq[0,1]^\omega$ are closed?]
Possibly, yes: try to use $F_n$ such that 
$F^{-1}_n\{0\}=A_n$ and $F:X\to[0,1]$ defined by
$F(x)=\sum_{n\in\omega}2^{-n}\cdot F_n(x)$, 
then $F$~is continuous, but do we get 
$A_n\in\alg F$ for all $n$?] 

\vskip+1em
\normalsize
}


\hide{
\section*{Images vs Hausdorff operations}
}

We turn to the problem of when classes of 
$\mathbf\Phi$-sets are preserved under maps 
in the images direction. Easily, the images 
of a~map~$F$ distributes over (even binary) 
intersections iff $F$~is one-to-one. Below 
we observe that the situation is less trivial 
if we consider intersections of families 
of sets directed by the converse inclusion.

A~map $F:X\to Y$ is {\em finite-to-one} iff 
$\ker F\subseteq\mathscr{P}_\omega(X)$, 
{\em closed-to-one} iff $\ker F\subseteq\mathscr{F}(X)$, 
and 
{\em compact-to-one} iff $\ker F\subseteq\mathscr{K}(X)$.

Trivially, any finite-to-one or proper~$F$ is 
compact-to-one. Also, any continuous closed 
compact-to-one~$F$ is perfect, and if $X$~is compact 
and $Y$~is Hausdorff then any continuous~$F$ is perfect 
(see, e.g., \cite{Engelking}, 3.2.7 and 3.1.12).

\hide{
\vskip+1em
\begin{footnotesize}
[It seems, now the following is superfluous:]
Let us also say that $F$~is {\em compact-to-one on 
compact sets} iff $F^{-1}\{y\}\cap A\in\mathscr K(X)$ 
for all $y\in Y$ and $A\in\mathscr K(X)$. Clearly,
if $X$~is Hausdorff then any compact-to-one~$F$ is 
compact-to-one on compact sets.

[A~more general terminology: 
$F:X\to Y$ is {\em $\mathscr S$-to-one} iff 
$\ker F\subseteq\mathscr S$, and {\em $\mathscr S$-to-one
on~$\mathscr T$} iff $A\cap B\in\mathscr S$ for all 
$A\in\ker F$ and $B\in\mathscr T$.]
\end{footnotesize}

\vskip+1em
}

Given a~partially ordered set $(I,\le)$, we shall say that 
a~family $(A_i)_{i\in I}$ of sets is {\em decreasing} 
iff $A_i\supseteq A_j$ for all $i\le j$. Considering 
below $\omega$ and $\omega^{<\omega}$ as sets of indices, 
we imply the natural ordering (i.e., by inclusion) of 
each of them.

The following result provides conditions under which
images distribute over intersections of directed 
decreasing families.


\begin{prop}\label{p: intersection vs image}
Let $F:X\to Y$. The equality
$
F\bigcap_{i\in I}A_i
=\bigcap_{i\in I}FA_i
$
holds for all directed $(I,\le)$ and
\begin{itemize}
\item[(i)]
all decreasing $(A_i)_{i\in I}$ in $\mathscr P(X)$
if $F$~is finite-to-one,
\item[(ii)]
all decreasing $(A_i)_{i\in I}$ 
in $(\mathscr F\cap\mathscr K)(X)$
if $F$~is closed-to-one.
\end{itemize}
\end{prop}

\hide{
\footnotesize

[In (i) we have ``iff'' in fact. 
In (ii) some doubts about ``iff'' arise\ldots
double check!!!]

\vskip+0.5em

[Also: perhaps, there should be a~statement 
general for items (i) and~(ii)?]

\vskip+1em
\normalsize
}

\begin{proof}
Since the inclusion
$
F\bigcap_{i\in I}A_i
\subseteq\bigcap_{i\in I}FA_i
$
holds always, we prove the converse inclusion.

(i). 
If $F$~is finite-to-one, let $(I,\le)$~be a~directed 
set and $(A_i)_{i\in I}$ a~family of nonempty sets 
such that $A_i\supseteq A_j$ if $i\le j$. Fix any 
$y\in\bigcap_{i\in I}FA_i$, i.e., $y$ such that 
$F^{-1}\{y\}\cap A_i\ne\emptyset$ for all $i\in I$, 
and show that $y\in F\bigcap_{i\in I}A_i$, i.e., 
that $F^{-1}\{y\}\cap\bigcap_{i\in I}A_i\ne\emptyset$. 
Since $F$~is finite-to-one, $|F^{-1}\{y\}|<\omega$, 
say, $F^{-1}\{y\}=\{x_k\}_{k<n}$ for some $n<\omega$. 
Toward a~contradiction, assume 
$F^{-1}\{y\}\cap\bigcap_{i\in I}A_i=\emptyset$, 
so for any $k<n$ there is $i_k\in I$ such that 
$x\notin A_{i_k}$. Since $(A_i)_{i\in I}$ is decreasing,
and so $\supseteq$-directed, there exists $i\in I$ 
such that $A_i\subseteq\bigcap_{k<n}A_{i_k}$. But then 
for every $k<n$ we have $x_k\notin A_i$, thus showing 
$F^{-1}\{y\}\cap A_i\ne\emptyset$; a~contradiction.

\hide{
\vskip+1em
\footnotesize

If $F$~is not finite-to-one, there exists $y\in Y$ 
with infinite $F^{-1}\{y\}$. Letting 
$(I,\le)=(\mathscr P_\omega(F^{-1}\{y\}),\supseteq)$ 
and $A_i=F^{-1}\{y\}\setminus i$ for each $i\in I$, 
we see that $(A_i)_{i\in I}$ is decreasing, 
$F\bigcap_{i\in I}A_i$ is empty while 
$\bigcap_{i\in I}FA_i$ is not.

\vskip+1em
\normalsize
}

(ii). 
If $F$~is closed-to-one, let $(I,\le)$~be a~directed set 
and $(A_i)_{i\in I}$ a~family of nonempty closed compact 
sets such that $A_i\supseteq A_j$ if $i\le j$. W.l.g.~let 
$I$~have a~least element, say,~$j$ (otherwise pick any 
$j\in I$ and consider $\{i\in I:i\ge j\}$ instead of~$I$),
hence $(A_i)_{i\in I}$ has the largest set~$A_j$.
If $y\in\bigcap_{i\in I}FA_i$ then the intersections 
$B_i=F^{-1}\{y\}\cap A_i$ are nonempty for all $i\in I$. 
Moreover, $B_i$~are closed subsets of the compact set~$A_j$ 
(since $F$~is closed-to-one and $A_i$~are closed) and form 
a~$\supseteq$-directed family (as $B_i$~are nonempty 
and $A_i$~form a~$\supseteq$-directed family). 
Any $\supseteq$-directed family of nonempty closed subsets 
of a~compact set has a~nonempty intersection, so pick 
an $x\in\bigcap_{i\in I}B_i$. We have 
$x\in F^{-1}\{y\}\cap\bigcap_{i\in I} A_i$ 
and hence $y\in F\bigcap_{i\in I}A_i.$
\hide{
\vskip+1em
\footnotesize

[old:]
If $F$~is not compact-to-one on compact sets, 
then there exist $y\in Y$ and $A\in\mathscr K(X)$ 
such that $F^{-1}\{y\}\cap A\notin\mathscr K(X)$.
Therefore there is a~$\supseteq$-directed family 
$(C_i)_{i\in I}$ of closed subsets of~$X$ such that sets 
$F^{-1}\{y\}\cap A\cap C_i$ are nonempty for all $i\in I$, 
but their intersection is empty. Let $A_i=A\cap C_i$. 
The~$A_i$ are compact. Moreover, $y\in\bigcap_{i\in I}FA_i$ 
however $y\notin F\bigcap_{i\in I}A_i.$ Indeed, 
$y\in F\bigcap_{i\in I}A_i$ iff 
$F^{-1}\{y\}\cap\bigcap_{i\in I}A_i\ne\emptyset$ iff
$F^{-1}\{y\}\cap\bigcap_{i\in I}(A\cap C_i)\ne\emptyset$ 
iff $\bigcap_{i\in I}(F^{-1}\{y\}\cap A\cap C_i)\ne\emptyset$, 
which is wrong.
\qquad[correct, then double check!!!]
\normalsize
}
\end{proof}

Clearly, in item~(ii) if $X$ is Hausdorff,
we may write just $\mathscr K(X)$ instead of 
$(\mathscr F\cap\mathscr K)(X)$. Also, if we 
consider only countable intersections (as they 
appear in $\omega$-ary Hausdorff operations), 
it suffices to assume only that $A_i$ are 
countably compact.

\hide{
\vskip+1em
\footnotesize

\begin{rmk}
Minor modifications of the proof above allow 
to state more general facts: the equality holds if 
$F$~is $(<\kappa)$-to-one and $(A_i)_{i\in I}$ is 
$\kappa$-directed by inverse inclusion (i.e., any 
$J\subseteq I$ of size~$<\kappa$ has an upper bound); 
and if $F$~is finally-$\kappa$-compact-to-one, 
$(A_i)_{i\in I}$ is $\kappa$-directed by inverse 
inclusion and $A_i$~are finally-$\kappa$-compact (i.e., 
any open covering includes a~subcovering of size~$<\kappa$). 
\qquad[also correct this on final $\kappa$-compactness!!!]
\end{rmk}

\vskip+1em
\normalsize
}


\begin{lmm}\label{l: hausd vs image}
Let $\varPhi$~be a~Hausdorff operation and $F:X\to Y$.
The equality
$
F\varPhi(A_s)_{s\in\omega^{<\omega}}=
\varPhi(FA_s)_{s\in\omega^{<\omega}}
$ 
holds
\begin{itemize}
\item[(i)] 
for all decreasing $(A_s)_{s\in\omega^{<\omega}}$ 
in $\mathscr P(X)$ 
if $F$~is finite-to-one,
\item[(ii)] 
for all decreasing $(A_s)_{s\in\omega^{<\omega}}$ 
in $(\mathscr F\cap\mathscr K)(X)$ 
if $F$~is closed-to-one.
\end{itemize}
\end{lmm}

\hide{ 
\footnotesize
[Does the dual operation~$\varPhi^*$ 
give the same for co-compacts? check!]
\normalsize
}

\begin{proof} 
Let $S\subseteq\omega^\omega$ be a~base of~$\varPhi$. 
Since the images of $F$ distribute over arbitrary unions 
and, by Proposition~\ref{p: intersection vs image}, over 
intersections of decreasing families of (closed compact) 
sets if $F$~is finite-to-one (closed-to-one), we have
\begin{align*}
F\varPhi(A_s)_{s\in\omega^{<\omega}}=
F\bigcup_{f\in S}\bigcap_{n\in\omega}A_{f\uhr n}
&=
\bigcup_{f\in S}F\bigcap_{n\in\omega}A_{f\uhr n}
\\
&=
\bigcup_{f\in S}\bigcap_{n\in\omega}FA_{f\uhr n}
=
\varPhi(FA_s)_{s\in\omega^{<\omega}},
\end{align*}
as required. 
\end{proof}

\begin{coro}\label{c: hausd vs image}
Let $\varPhi$~be a~Hausdorff operation, $F:X\to Y$, 
and $\mathscr S$~a~class of sets such that
$\mathscr S(X)$~is closed under finite intersections 
and $F$ preserves~$\mathscr S$, i.e.,
$F\mathscr S(X)\subseteq\mathscr S(Y)$.
Then $F$ preserves~$\varPhi(\mathscr S)$, i.e.,
$F\varPhi(\mathscr S,X)\subseteq\varPhi(\mathscr S,Y)$,
whenever
\begin{itemize}
\item[(i)] 
$F$ is finite-to-one, or
\item[(ii)] 
$F$ is closed-to-one and 
$\mathscr S(X)\subseteq(\mathscr F\cap\mathscr K)(X)$.
\end{itemize}
\end{coro}

\begin{proof}
If $\mathscr S(X)$ is closed under finite intersections,
then every $(A_s)_{s\in\omega^{<\omega}}$ in $\mathscr S(X)$ 
can be replaced with a~decreasing 
$(B_s)_{s\in\omega^{<\omega}}$ in $\mathscr S(X)$ so that
$$
\mathbf\Phi(A_s)_{s\in\omega^{<\omega}}=
\mathbf\Phi(B_s)_{s\in\omega^{<\omega}}
$$
by letting $B_{f\uhr n}=\bigcap_{k\le n}A_{f\uhr k}$.
Now the claim follows from Lemma~\ref{l: hausd vs image}.
\end{proof}

E.g., as each of $\mathscr F,\mathscr G,\mathscr Z$ 
on a~given~$X$, as well as $\mathscr K$ on Hausdorff 
spaces~$X$, is closed under finite intersections, 
we see: if $F$~is closed and finite-to-one then it 
preserves $\varPhi(\mathscr F)$; if $F$~is open and 
finite-to-one then it preserves $\varPhi(\mathscr G)$;
if $X,Y$ are Hausdorff, $X$~is compact (and so normal), 
and $F$~is continuous (and so perfect), then $F$ 
preserves $\varPhi(\mathscr S)$ where $\mathscr S$ 
is each of $\mathscr F,\mathscr G,\mathscr Z,\mathscr K$.


\hide{
\section*{Reduction, separation}
}

Now we combine our previous results to transfer 
the reduction and separation properties in the 
pre-image direction.

\begin{prop}\label{p: preimage vs red}
Let $\varPhi$ be a~Hausdorff operation 
and $\mathscr S$~a~class such that for any 
$(A_n,B_n)_{n\in\omega}$ in $\mathscr S(X)$
there are $Y$ and $F:X\to Y$ such that 
\begin{itemize}
\item[(a)] 
$F^{-1}$ preserves~$\mathscr S$, 
\item[(b)] 
$(A_n,B_n)_{n\in\omega}$ is in $\alg F$, 
and 
\item[(c)] 
$
F\varPhi(A_n)_{n\in\omega},
F\varPhi(B_n)_{n\in\omega}
$ 
are reduced (separated) by 
sets in $\varPhi(\mathscr S,Y)$.
\end{itemize}
Then $\varPhi(\mathscr S,X)$ has 
the reduction (separation) property.
\end{prop}

\begin{proof}
Prove, e.g., reduction. 
Pick any $A,B$ in $\varPhi(\mathscr S,X)$ and 
$(A_n)_{n\in\omega}$, $(B_n)_{n\in\omega}$ 
in $\mathscr S(X)$ such that
$A=\varPhi(A_n)_{n\in\omega}$, 
$B=\varPhi(B_n)_{n\in\omega}$.
Let $Y$ and $F:X\to Y$ be such that 
$F^{-1}$ preserves~$\mathscr S$, 
all the sets $A_n,B_n$ are in $\alg F$, 
and the sets $FA=F\varPhi(A_n)_{n\in\omega}$,
$FB=F\varPhi(B_n)_{n\in\omega}$ are reduced 
by some sets in $\varPhi(\mathscr S,Y)$, 
i.e., there exist $C,D$ in $\varPhi(\mathscr S,Y)$ 
such that 
$$
C\subseteq FA,\;\;
D\subseteq FB,\;\;
C\cap D=\emptyset,
\;\;\text{and}\;\;
C\cup D=(FA)\cup(FB).
$$
As $F^{-1}$ preserves~$\mathscr S$, 
it preserves $\varPhi(\mathscr S)$ by 
Corollary~\ref{c: hausd vs preimage}, so 
$F^{-1}C,F^{-1}D$ are in $\varPhi(\mathscr S,X)$.
Moreover, we have: 
$$
F^{-1}C\subseteq F^{-1}FA=A 
\text{ and }
F^{-1}C\subseteq F^{-1}FB=B
$$ 
(where the equalities hold
by Lemma~\ref{l: algebra of preimages}
since the $A_n,B_n$ are in $\alg F$ 
and so $A,B$ are also in $\alg F$),
also $F^{-1}C\cap F^{-1}D=\emptyset$, 
and finally,
\begin{align*}
F^{-1}C\cup F^{-1}D=
F^{-1}(C\cup D)
&=
F^{-1}((FA)\cup(FB))
\\
&=
(F^{-1}FA)\cup(F^{-1}FB)=
A\cup B
\end{align*}
(as pre-images distribute over unions and 
$A,B$ are in $\alg F$). This proves reduction 
in $\varPhi(\mathscr S,X)$, as required. 
\end{proof}

\begin{prop}\label{p: image preimage vs red}
Let $\varPhi$ be a~Hausdorff operation and 
$\mathscr S$~a~class of sets. 
\begin{itemize}
\item[(i)] 
If $\mathscr S(X)$ is closed under finite intersections 
and such that 
for any $(A_n)_{n\in\omega}$ in $\mathscr S(X)$
there are $Y$ and a~finite-to-one $F:X\to Y$ such that 
\begin{itemize}
\item[(a)] 
$F$ and $F^{-1}$ preserve~$\mathscr S$, 
\item[(b)] 
$(A_n)_{n\in\omega}$ is in $\alg F$, and 
\item[(c)] 
$\varPhi(\mathscr S,Y)$ has the reduction 
(separation) property, 
\end{itemize}
then $\varPhi(\mathscr S,X)$ 
has the same property.
\item[(ii)] 
The same remains true assuming  
$\mathscr S(X)\subseteq(\mathscr F\cap\mathscr K)(X)$ 
and that maps~$F$ are (not necessarily finite-to-one
but) closed-to-one.
\end{itemize}
\end{prop}

\begin{proof}
Since $\mathscr S(X)$ is closed under 
finite intersections (and is included into 
$(\mathscr F\cap\mathscr K)(X)$ in~(ii)) and 
as $F$ preserves~$\mathscr S$ and is finite-to-one 
(or closed-to-one in~(ii)), it preserves 
$\varPhi(\mathscr S)$ by 
Corollary~\ref{c: hausd vs image}, so 
$FA,FB$ are in $\varPhi(\mathscr S,Y)$, and so 
by reduction in $\varPhi(\mathscr S,Y)$, they are 
reduced by some sets in $\varPhi(\mathscr S,Y)$. 
Now we are in position to apply 
Proposition~\ref{p: preimage vs red}
thus getting the same conclusion. 
\end{proof}


\begin{lmm}\label{l: zero tych}
Let $Y$~be Tychonoff and $X\subseteq Y$. 
Then $\mathscr Z(X)=\mathscr Z(Y)\uhr X$ 
and consequently, 
for any Hausdorff operation~$\varPhi$, 
\begin{itemize}
\item[(i)] 
$\varPhi(\mathscr Z,X)=\varPhi(\mathscr Z,Y)\uhr X$,
\item[(ii)] 
if $\varPhi(\mathscr Z,Y)$ has reduction 
(separation) then $\varPhi(\mathscr Z,X)$
has the same property.
\end{itemize}
\end{lmm}

\begin{proof}
The inclusion $\mathscr Z(Y)\uhr X\subseteq\mathscr Z(X)$
holds for arbitrary spaces $X,Y$: if a~continuous 
map $F:Y\to[0,1]$ witnesses that $B\in\mathscr Z(Y)$, i.e., 
$F^{-1}\{0\}=B$, then its restriction $F\uhr X$ witnesses 
that $B\cap X\in\mathscr Z(X)$. To verify the converse 
inclusion $\mathscr Z(X)\subseteq\mathscr Z(Y)\uhr X$, 
we use that $Y$ (and hence,~$X$) is Tychonoff.

We homeomorphically embed, first, $Y$ into $\upbeta Y$, 
the Stone--{\v C}ech compactification of~$Y$ (which exists 
for any Tychonoff space), and then, $\upbeta Y$ into the 
Tychonoff cube $[0,1]^\kappa$ for a~suitable~$\kappa$ 
(see, e.g., \cite{Engelking}, 2.3.23; it suffices to let 
$\kappa$ equal to the cardinality of all continuous maps 
of $Y$ into $[0,1]$, as follows from one of possible 
definitions of~$\upbeta Y$). Any continuous $F:Y\to[0,1]$ 
extends to the continuous $G:\upbeta Y\to[0,1]$ by the 
main property of the Stone--{\v C}ech compactification, 
and then, as $\upbeta Y$ is closed in the normal space 
$[0,1]^\kappa$, to a~continuous $H:[0,1]^\kappa\to[0,1]$ 
by the Tietze--Urysohn extension theorem (\cite{Engelking}, 
3.6.3 and 2.1.8, respectively). Therefore, if a~continuous 
$F:Y\to[0,1]$ witnesses that $B\in\mathscr Z(Y)$, 
then its continuous extension~$H$ witnesses that 
$C\in\mathscr Z([0,1]^\kappa)$ with $B=C\cap Y$, thus proving
$\mathscr Z(Y)\subseteq\mathscr Z([0,1]^\kappa)\uhr Y$.

Putting both inclusions together, we get 
$\mathscr Z(Y)=\mathscr Z([0,1]^\kappa)\uhr Y$. But as  
$Y$~is arbitrary, the same equality holds for $X\subseteq Y$, 
whence the required $\mathscr Z(X)=\mathscr Z(Y)\uhr X$ 
easily follows. Now the ``consequently'' part follows 
by Corollary~\ref{c: hausd in subspace}(ii). 
\end{proof}


\section{Main results}

The following theorem is the main result 
of this paper:

\begin{thm}\label{t: red tych}
Let $X$ be a~Tychonoff space and
$\varPhi$ a~Hausdorff operation. If 
$\varPhi(\mathscr F,\mathbb R)$ has 
the reduction (separation) property, then
$\varPhi(\mathscr Z,X)$ has the same property. 
\end{thm}

\begin{proof}
First show that the claim is true for 
all Tychonoff cubes $X=[0,1]^\kappa$. For 
$\kappa=\omega$ this is trivial since 
$[0,1]^\omega$ is Polish. For arbitrary~$\kappa$,
let us verify that the assumptions of 
Proposition~\ref{p: image preimage vs red}(ii) 
are met with $\mathscr S=\mathscr Z$ and 
$Y=[0,1]^\omega$ common for all $(A_n)_{n\in\omega}$ 
in $\mathscr S(Y)$, 
i.e., in $\mathscr Z([0,1]^\omega)$.

Indeed, $\mathscr Z([0,1]^\omega)$ is closed under 
finite intersections. If $(A_n)_{n\in\omega}$ is 
in $\mathscr Z([0,1]^\kappa)$, then 
Proposition~\ref{p: zero sets vs diag prod} gives 
a~continuous map $F:[0,1]^\kappa\to[0,1]^\omega$ 
such that $\varPhi(A_n)_{n\in\omega}\in\alg F$. 
As $[0,1]^\kappa$ is compact Hausdorff,
$
\mathscr Z([0,1]^\kappa)\subseteq
(\mathscr F\cap\mathscr K)([0,1]^\kappa)=
\mathscr K([0,1]^\kappa),
$
and moreover, $F$~is perfect (as a~continuous 
map of a~compact space into a~Hausdorff space), 
whence it is easy to see that $F$ and $F^{-1}$ 
preserve~$\mathscr Z$. Therefore, once we have reduction 
(separation) in $\varPhi(\mathscr Z,[0,1]^\omega)$, 
or equivalently, in $\varPhi(\mathscr F,\mathbb R)$, 
we are able to apply 
Proposition~\ref{p: image preimage vs red}(ii), 
thus getting the same property in 
$\varPhi(\mathscr Z,[0,1]^\kappa)$.

Now let $X$~be an arbitrary Tychonoff space. 
Pick any~$\kappa$ such that $X$~can be identified 
with a~subspace of $[0,1]^\kappa$. Then 
$\varPhi(\mathscr Z,X)$ has reduction (separation) 
whenever $\varPhi(\mathscr Z,[0,1]^\kappa)$ has 
the same property by Lemma~\ref{l: zero tych}, 
and hence, whenever $\varPhi(\mathscr F,\mathbb R)$ 
has this property.

The proof is complete.
\end{proof}


In particular, the Borel and projective classes 
(as they were defined in the beginning of our paper 
for arbitrary spaces) generated from zero sets in 
Tychonoff spaces form the same pattern of reduction 
and separation as they do in the real line:

\begin{coro}\label{c: period tych}
Let $X$ be a~Tychonoff space. Then:
\begin{itemize}
\item[(i)] 
for all $\alpha<\omega_1$, $\alpha>1$,
$\mathbf\Sigma^{0}_{\alpha}(\mathscr Z,X)$ 
have the reduction property while 
$\mathbf\Pi^{0}_{\alpha}(\mathscr Z,X)$ 
have the separation property,
\item[(ii)] 
$\mathbf\Pi^{1}_{1}(\mathscr Z,X)$ and 
$\mathbf\Sigma^{1}_{2}(\mathscr Z,X)$ 
have the reduction property while 
$\mathbf\Sigma^{1}_{1}(\mathscr Z,X)$ 
and $\mathbf\Pi^{1}_{2}(\mathscr Z,X)$ 
have the separation property,
\item[(iii)] 
under $\PD$, for all $n<\omega$, $n>0$, 
$\mathbf\Sigma^{1}_{2n}(\mathscr Z,X)$ 
and $\mathbf\Pi^{1}_{2n+1}(\mathscr Z,X)$ 
have the reduction property while 
$\mathbf\Sigma^{1}_{2n+1}(\mathscr Z,X)$ 
and $\mathbf\Pi^{1}_{2n}(\mathscr Z,X)$ 
have the separation property.
\end{itemize}
\end{coro}

\begin{proof}
As well-known, if $X$ is~$\mathbb R$ 
(or another Polish space) and hence 
$\mathscr Z(X)$ is equal to $\mathscr F(X)$, then  
items (i)--(iii) hold. Moreover, all Borel classes 
$\mathbf\Sigma^{0}_{\alpha}(\mathscr F,\mathbb R)$ 
have the pre-well-ordering property, so all they 
have reduction while the dual classes 
$\mathbf\Pi^{0}_{\alpha}(\mathscr F,\mathbb R)$ 
have separation (see~\cite{Moschovakis}, p.~37);
the projective classes 
$\mathbf\Pi^{1}_{1}(\mathscr F,\mathbb R)$ and 
$\mathbf\Sigma^{1}_{2}(\mathscr F,\mathbb R)$ 
have pre-well-ordering, so they have reduction 
while $\mathbf\Sigma^{1}_{1}(\mathscr F,\mathbb R)$ 
and $\mathbf\Pi^{1}_{2}(\mathscr F,\mathbb R)$ 
have separation;
and under~$\PD$, all projective classes 
$\mathbf\Sigma^{1}_{2n}(\mathscr F,\mathbb R)$ and 
$\mathbf\Pi^{1}_{2n+1}(\mathscr F,\mathbb R)$ have 
pre-well-ordering (the fact known as the First 
Periodicity Theorem), so all they have reduction 
while $\mathbf\Pi^{1}_{2n}(\mathscr F,\mathbb R)$ 
and $\mathbf\Sigma^{1}_{2n+1}(\mathscr F,\mathbb R)$ 
have separation (see \cite{Kechris}, 
\cite{Moschovakis}, or~\cite{Kanamori}). 
Now apply Theorem~\ref{t: red tych}.
\end{proof}

Certainly, Theorem~\ref{t: red tych} allows 
to establish further corollaries in the same way. 
E.g., under $\sigma$-$\PD$, the $\sigma$-Projective 
Determinacy, Corollary~\ref{c: period tych}(ii) 
extends to $\sigma$-projective classes generated 
by zero sets in Tychonoff spaces~$X$: classes
$\mbf\Sigma^{1}_{2\alpha}(\mathscr Z,X)$ for all 
$\alpha>0$ and $\mbf\Pi^{1}_{2\alpha+1}(\mathscr Z,X)$ 
have reduction while dual 
$\mbf\Sigma^{1}_{2\alpha+1}(\mathscr Z,X)$ 
and $\mbf\Pi^{1}_{2\alpha}(\mathscr Z,X)$ 
have separation. On the other hand, under $V=L$, 
$\mathbf\Pi^{1}_{1}(\mathscr Z,X)$ and 
$\mathbf\Sigma^{1}_{n}(\mathscr Z,X)$ for all $n>1$ 
have reduction, and the dual classes have separation.  
Finally, let us point out that some close principles, 
like the second separation or the multiple reduction 
properties, can be established for corresponding 
classes in Tychonoff spaces following the same approach.

\hide{

\newpage

\footnotesize

\begin{q}
Propositions \ref{p: preimage vs red}, 
\ref{p: image preimage vs red} transfer 
reduction (separation) in the pre-image direction. 
Find natural sufficient conditions to transfer 
these properties in the image direction. 
Do not continuous (perfect?) open maps provide 
a~useful tool for this? Then we could try to apply 
Ponomarev's result on continuous open maps for 
establishing reduction (separation) in arbitrary 
spaces. (The result states that any space of 
weight~$\kappa$ is a~continuous open image of 
a~subspace of the Baire space $\kappa^\omega$.) 
\end{q}

\begin{q}\label{q: hausdorff ord}
Given $S\subseteq\omega^\omega$, let $\varPhi_S$~denote 
the Hausdorff operation with base~$S$. 
For $S,T\subseteq\omega^\omega$, define $S\le_\Ha T$ iff
$
\varPhi_S(\mathscr F,\omega^\omega)\subseteq
\varPhi_T(\mathscr F,\omega^\omega).
$
What is $\le_\Ha$? a~relationship to~$\le_\W$ 
(the Wadge ordering)?
\end{q}

\begin{q}\label{q: t0 spaces}
Let $X$ be a~connected two-point space, e.g., 
$X=2$ with $\tau=\{\emptyset,\{0\},2\}$. 
Is there a~perfect map of the product $X^\omega$ 
onto $[0,1]$? Perhaps, the map taking every~$f$ 
such that there is~$m$ with 
$f(m)=0$ and $f(n)=1$ for all $n>m$, to $g$~such 
that $g(n)=f(n)$ for all $n<m$, $g(m)=1$, and 
$g(n)=0$ for all $n>m$, and fixing all other~$f$'s? 
Clearly, it is two-to-one, but is it continuous? 
Check!

The motivation of the question: If yes, we have 
a~perfect map of any $X^\kappa$ to $[0,1]$, whence
we get the same result
for any $T_0$-space since $X^\kappa$ is universal 
for $T_0$-spaces of weight~$\kappa$.

[In fact, an analogous statement is true for all 
spaces (not only~$T_0$), see [Engelking], 2.3.I.]
\end{q}

\begin{q}
Can we refine the obtained results to the countable 
reduction property, and then apply \cite{Choban 1989}, 
Theorem~22.2, to get some selection results? If yes, 
does this mean that a~certain uniformization follows 
from reduction? (The latter is not obvious since that 
theorem says on multivalued maps into Polish spaces.) 
\end{q}

\begin{q}
Given $f:\omega\setminus1\to2$, let 
$\varphi_f$~assert that $\mbf\Sigma^{1}_{n}$~has 
reduction if $f(n)=0$, and separation otherwise. 
For which $f$ the assertion~$\varphi_f$ is 
(relatively) consistent? 
Of course, we should have $f(1)=1$ and $f(2)=0$. 

A~variant of the question: let $\psi_f$~assert 
that the pre-well-ordering property holds in 
$\mbf\Sigma^{1}_{n}$ if $f(n)=0$, and in 
$\mbf\Pi^{1}_{n}$ otherwise. For which $f$ 
the assertion~$\psi_f$ is (relatively) consistent? 
Again, we should have $f(1)=1$ and $f(2)=0$, and 
moreover, if $f(n)=1$ then $f(n+1)=0$ by 
Novikov--Moschovakis result, see \cite{Kanamori}, 
29.9. 

Are there some extra necessary conditions?
\end{q}

\begin{q} 
Given $S\subseteq\omega^\omega$ and 
$(A_f)_{f\in S}$, let 
$
\mathbf G_{f\in S}A_f=
\{x:\Game f\in S\;x\in A_f\},
$
where $\Game f$ informally means 
$
\exists f(0)\,\forall f(1)\,
\exists f(2)\,\forall f(3)\,
\ldots 
$
(and formally is defined via winning strategies) 
and $\Game f\in S$ means $\Game f\wedge f\in S$. 
Thus $\mathbf G$~is an operation on sets indexed 
by sequences.

1. 
Is this operation analytical (in sense of 
\cite{Kantorovich Livenson})? 
It seems, it cannot be an $\omega$-ary 
$\delta\/s$-operation because, by Burgess' 
theorem, we have 
$
\mbf G\mbf\Delta^{0}_1=
\mbf\Sigma^{0}_{<\omega_1}
$
(=~Borel), which is impossible for $\omega$-ary 
$\delta\/s$-operations, though possible for 
$\omega_1$-ary ones, see~\cite{Dasgupta}.

2. 
Find conditions on a~map~$F$ under which the 
operation~$\mbf G$ is preserved in the pre-image 
direction (always?) and in the image direction 
(is it sufficient that $F$~preserves intersections?).
\end{q}

\vskip+1em
\normalsize

\newpage

\newpage

\begin{footnotesize}

[old text starts:

\section*{Consequences on non-degenerate classes}

[The following result seems redundant, it follows from
a~general criteria on ``non-trivial'' spaces, see 
\cite{Choban 1989}.]

\begin{tm} 
Let $\varPhi$~be a~Hausdorff operation, and let 
$X$~include a~subspace of one of the following types:
\begin{itemize}
\item[(i)]
a~perfect pre-image of a~complete metric 
space without isolated points, 
\item[(ii)]
a~perfectly normal compact subspace 
without isolated points,
\item[(iii)]
a~space homeomorphic to $\scc Y\setminus Y$
for an infinite discrete~$Y$. 
\end{itemize}
Then there exists $A\subseteq X$ such that 
$A$~is obtained by~$\varPhi$ from closed sets and 
$X\setminus A$~is not. 
In particular, all Borel and projective classes 
generated by closed sets in~$X$ are non-empty.

The same holds for open sets instead closed ones. 
\end{tm}

\begin{proof}
(i). 
Since Theorem the theorem holds in the Cantor space 
and any complete metric space without isolated points
includes a~copy of the Cantor set [a~reference?], 
the theorem follows from Lemma [on subsets] and 
Corollary [on perfect pre-images of compact spaces]. 

(ii). 
By Lemma [on subsets], it suffices to consider 
the case when $X$~is a~perfectly normal space.
By~\cite{Chernavski}, for any such~$X$ there exists 
a~perfect map of $X$ onto the unit interval of~$\mathbb R$.
Then the statement follows from [the main theorem above].

(iii). 
By~\cite{Przymusinsky}, 
any perfectly normal compact space is 
a~continuous image of $\scc\omega\setminus\omega$. 
So if $G$~is a~continuous map of 
$\scc\omega\setminus\omega$ onto the Cantor set
and $H$~is the map of $\scc Y\setminus Y$ onto
$\scc\omega\setminus\omega$ which is the restriction
of the continuous extension of a~map of $Y$ onto~$\omega$,
then $F=G\circ H$ is a~continuous, and so perfect, map 
of $\scc Y\setminus Y$ onto the Cantor set. 
Now apply~(i).
\end{proof}

Item~(ii) improves Ponomarev's result on Borel sets 
in such spaces~\cite{Ponomarev}.

[For (iii), the existence of a~continuous map of 
$\scc\omega\setminus\omega$ onto the Cantor set
(or the interval $[0,1]$ of the real line) 
is obvious : extend any map of the discrete 
space~$\omega$ to a~countable dense subset.]

[NB 
The space $\scc\omega\setminus\omega$ 
is neither perfectly normal nor contains a~copy of 
the Cantor set, hence, it does not cover neither by
Ponomarev theorem nor by Hausdorff theorem.]

\vskip+1em


\begin{qs}
1. 
What are $X$ that are perfect pre-images of: 
\begin{enumerate}
\item[(a)]
complete metric spaces without isolated points?
\item[(b)]
the unit interval of the real line?
\item[(c)]
subspaces of the unit interval of the real line?
\item[(d)]
sets of reals in $L(\mathbb R)$?
\end{enumerate}
Note that by~\cite{Frolik}, 
perfect pre-images of all complete metric spaces are 
precisely all paracompact {\v C}ech-complete spaces, 
while perfect pre-images of all metric spaces are 
precisely all paracompact p-spaces 
(see also~\cite{Engelking}).

2. 
What are $X$ such that 
for any closed (compact) $A\subseteq X$ 
there exists a~perfect map~$F$ onto 
the space in items (a), (b),~(c) in question~1
with $F^{-1}F\image A=A$?

3. 
Does any complete metric space without isolated points:
\begin{enumerate}
\item[(a)]
contain a~copy of the Cantor set?
\item[(b)]
perfectly map onto the unit interval of the real line?
\end{enumerate}

4. 
Find any example of a~perfectly normal compact space 
without subspaces homeomorphic to the Cantor set.
\end{qs}

\vskip+1em
old text ends.]

\end{footnotesize}

\newpage

\begin{footnotesize}

\section*{The $\sigma$-projective hierarchy: 
pre-well-ordering, reduction, separation}

\qquad
[The classes defined below are called 
``$\sigma$-projective'' in~\cite{Kechris} and 
``weakly projective'' in~\cite{Di Prisco et al 1982}; 
``hyperprojective'' means certain more 
complex classes; see also 
\cite{Di Prisco et al 1982},
\cite{Di Prisco et al 1987}. 
]

\vskip+1em

Let $\mathscr S$~be a~class of families 
of subspaces of topological spaces. We define 
the $\sigma$-{\it projective\/} hierarchy 
generated by $\mathscr S$ as follows. 
For all~$X$, define by recursion on ordinals 
$\alpha<\omega_1$:
\begin{align*}
\mbf\Sigma^{1}_{0}(\mathscr S,X)
&=
\mathscr S(X),
\\
\mbf\Pi^{1}_{\alpha}(\mathscr S,X)
&=
\bigl\{
X\setminus A:
A\in\mbf\Sigma^{1}_{\alpha}(\mathscr S,X)
\bigr\},
\\
\mbf\Sigma^{1}_{\alpha+1}(\mathscr S,X)
&=
\bigl\{
\pr_X(A):A\in
\mbf\Pi^{1}_{\alpha}(\mathscr S,X\times\omega^\omega) 
\bigr\}, 
\\
\mbf\Sigma^{1}_{\alpha}(\mathscr S,X)
&=
\bigl\{
{\textstyle\bigcup_{n<\omega}}A_n:
(A_n)_{n<\omega}\in
\bigl({\textstyle\bigcup_{\beta<\alpha}}
\mbf\Pi^{1}_{\beta}(\mathscr S,X)\bigr)^\omega
\bigr\}
\text{ if $\alpha$ is limit},
\\
\mbf\Delta^{1}_{\alpha}(\mathscr S,X)
&=
\mbf\Pi^{1}_{\alpha}(\mathscr S,X)
\cap
\mbf\Sigma^{1}_{\alpha}(\mathscr S,X).
\end{align*}

Note that 
$\mbf\Sigma^{1}_{\alpha}(\mathscr S,X)
=\mbf\Pi^{1}_{\alpha}(\mathscr S,X)$
whenever $\alpha>0$ is limit. 
[No, this seems wrong!!! the true fact is
$\mbf\Sigma^{1}_{<\alpha}(\mathscr S,X)
=\mbf\Pi^{1}_{<\alpha}(\mathscr S,X)$, 
if we define
$\mbf\Sigma^{1}_{\alpha}(\mathscr S,X)$
as 
$\bigcup_{\beta<\alpha}
\mbf\Pi^{1}_{\beta}(\mathscr S,X)$ and 
$\mbf\Pi^{1}_{<\alpha}(\mathscr S,X)$ 
as the dual class. Check and correct!
]

Under~$\AC_\omega$, the families
$$
\bigcup_{\alpha<\omega_1}
\mbf\Pi^{1}_{\alpha}(\mathscr S,X)=
\bigcup_{\alpha<\omega_1}
\mbf\Sigma^{1}_{\alpha}(\mathscr S,X)
$$
are the smallest $\sigma$-algebras of sets in~$X$ 
which include $\mathscr S(X)$ and closed under 
projections of sets in $X\times\omega^\omega$
onto~$X$. 
[This can be written by 
$
\mbf\Pi^{1}_{<\omega_1}(\mathscr S,X)=
\mbf\Sigma^{1}_{<\omega_1}(\mathscr S,X).
$
]

Note that the Hausdorff operation $\varPhi_S$ with 
a~base $S\subseteq\omega^\omega$ gives the same sets 
that the projection $\pr_X$ of $X\times S$ onto~$X$ 
[and continuous images of~$S$ in~$X$? check!]. 
Moreover, $\varPhi_S$ applied to sets in $\mathscr F(X)$
gives the same that $\pr_X$ applied to sets in 
$\mathscr F(X\times S)$. See \cite{Kantorovich Livenson}, 
Theorem~[?? write up].

[These classes are defined, e.g., in 
\cite{Kechris}, 39.15--39.18; also they are studied in 
\cite{Di Prisco et al 1982},~\cite{Di Prisco et al 1987}.
Check their properties established there!!]

\vskip+0.5em

A~fundamental result by Kantorovich and Livenson 
states that, whenever $X$ is a~Suslin space (i.e., 
a~continuous image of $\omega^\omega$), then
the class~$\mathscr S$ of projections onto $X$ of 
sets in $\varPhi(\mathscr F,X\times\omega^\omega)$ 
is itself of form $\varPhi_S(\mathscr F,X)$ for some 
set $S\in\varPhi(\mathscr F_\sigma,\omega^\omega)$
(see \cite{Kantorovich Livenson}, p.~264, the 
Fundamental Theorem on Projections). It follows 
by induction on~$\alpha$ that each of the 
$\sigma$-projective classes is of form 
$\varPhi(\mathscr F,X)$ for an appropriate 
operation~$\varPhi$.

\begin{q} 
Projective and $\sigma$-projective sets are 
exactly sets parametrically definable in 
$(\omega,\omega^\omega,0,\Sc,+,\,\cdot\,,\Ap)$ (where 
$\Sc$ and $\Ap$ are the successor and application 
operations) by formulas of $\mathscr L_{\omega,\omega}$ 
and $\mathscr L_{\omega_1,\omega}$, respectively 
(see \cite{Di Prisco et al 1982}, Theorems 0 and~1).

Can we have a~similar result for such sets in 
arbitrary spaces~$X$ (perhaps, with predicate symbols 
for each $S\in\mathscr S$ or something like)?
\end{q}

\begin{q}
Characterize classes $\mathscr S(X)\subseteq\mathscr P(X)$ 
of form $\varPhi(\mathscr Z,X)$ for some operation~$\varPhi$. 

(Compare an analogous question with $\mathscr F\cap\mathscr G$ 
instead of $\mathscr Z$, which is answered in \cite{Dougherty}, 
Proposition~1.4 as follows: For any space~$X$, a~class 
$\mathscr S(X)\subseteq\mathscr P(X)$ is of form 
$\varPhi(\mathscr F\cap\mathscr G,X)$ for some $\kappa$-ary 
operation~$\varPhi$ iff it coincides with 
$\{A\subseteq X:A\leq_\W T\}$ for some $T\subseteq 2^\kappa$
(where $\leq_\W$~is the Wadge reducibility and 
$2^\kappa$ is the Cantor cube of weight~$\kappa$). 
For other characterizations, also see~\cite{Dougherty}.) 
\end{q}

\begin{q}
Can the Fundamental Theorem on Projections be 
expanded (mutatis mutandi) to Tychonoff spaces~$X$, 
or at least, to all Tychonoff cubes $[0,1]^\kappa$? 
with one and the same base $S$ for all such~$X$'s? 
\end{q}

\begin{q}
Does the Fundamental Theorem on Projections remain true
(mutatis mutandi) for the game quantifier~$\Game$ 
instead of the usual quantifier~$\exists$, i.e., for 
the ``game projection'' instead of the usual projection? 
(And if yes, can it be expanded to Tychonoff spaces~$X$?) 
\end{q}


\vskip+1em

Note that the First Periodicity Theorem extends 
to $\sigma$-projective sets in Polish spaces, see
\cite{Kechris} or~\cite{Di Prisco et al 1982}.  
For brevity, below we write just 
$\mbf\Sigma^{1}_{\alpha}$ instead of 
$\mbf\Sigma^{1}_{\alpha}(\mathscr F,\omega^\omega)$.

[The next lemma is redundant since we may use 
$\sigma$-$\PD$ immediately instead of checking that 
it follows from $\AD^{L(\mathbb R)}$.]

\begin{lm} 
All the $\sigma$-projective sets in~$\omega^\omega$ 
are in $L(\mathbb R)$.
\end{lm}

\begin{proof}
Use that countable unions are encoded by reals. 
[More details?]
\end{proof}

[The next lemma is well-known, see 
\cite{Di Prisco et al 1982}, Proposition~8,
or \cite{Kechris}, Exercise~39.18. 
]

\begin{lm}
Assume $\AD^{L(\mathbb R)}$. 
In $\omega^\omega$, for all $\alpha<\omega_1$, 
the classes $\mbf\Sigma^{1}_{2\alpha}$ and 
$\mbf\Pi^{1}_{2\alpha+1}$ have the pre-well-ordering 
property. 
Consequently, these classes have the reduction property, 
while the dual classes $\mbf\Sigma^{1}_{2\alpha+1}$ 
and $\mbf\Pi^{1}_{2\alpha}$ have the separation property.
\end{lm}

\begin{proof}
Let us note first that if $\alpha>0$ is limit (and so 
$2\alpha=\alpha$) then the class $\mbf\Sigma^{1}_{\alpha}$ 
clearly has the pre-well-ordering property without 
any determinacy assumptions. Indeed, any set~$A$ in
$\mbf\Sigma^{1}_{\alpha}$ can be represented by 
a~countable union of {\it disjoint\/} sets in
$\bigcup_{\beta<\alpha}\mbf\Pi^{1}_{\beta}$: 
if $A=\bigcup_{n<\omega}A_n$ for some $A_n$ in 
$\bigcup_{\beta<\alpha}\mbf\Pi^{1}_{\beta}$, 
letting $A'_n=A_n\setminus\bigcup_{i<n}A_i$, 
we get $A=\bigcup_{n<\omega}A_n$ with $A'_n$ in
$\bigcup_{\beta<\alpha}\mbf\Pi^{1}_{\beta}$.
Then the map~$\rho$ of $A$ into ordinals defined
by letting $\rho(x)=n$ whenever $x\in A'_n$, 
is a~norm in $\mbf\Sigma^{1}_{\alpha}$. 
[write this fact as a~separate lemma?]

Then we can prove the general statement by 
induction on~$\alpha$, using at successor steps 
[the lemma above] and two following facts 
(see \cite{Kanamori}, Theorems 29.9 and 29.13;
or \cite{Kechris}, Theorem~39.1, 
or~\cite{Moschovakis}):
\begin{enumerate}
\item[(i)] 
if $\forall\mbf\Gamma\subseteq\mbf\Gamma$ 
and $\mbf\Gamma$ has the pre-well-ordering property, 
then so does $\exists\mbf\Gamma$,
\item[(ii)] 
assuming $\DC$ and $\Det(\btu\mbf\Gamma)$, 
if $\exists\mbf\Gamma\subseteq\mbf\Gamma$ 
and $\mbf\Gamma$ has the pre-well-ordering property,
then so does $\forall\mbf\Gamma$.
\end{enumerate}
[In~\cite{Kechris}, Theorem~39.1, this version of 
the First Periodicity Theorem relates only to those 
classes~$\mathbf\Gamma$ of sets in Polish spaces that 
are closed under continuous pre-images and projections.]
\end{proof}

\begin{q}
1. 
Does this First Periodicity Theorem (i.e., items (i)
and~(ii) above) remain true for more general classes? 
e.g., for effective classes of sets of reals? 
[this may be well known, check!!!]

2. 
The same question for the game quantifier~$\mathsf{G}$ instead 
of the usual quantifier (like \cite{Kechris}, Proposition~39.6
and Theorem~39.7).
\end{q}

\newpage 

\begin{q}
Does $\PD$ imply $\Det(\mbf\Sigma^{1}_{\omega})$?
More generally, what is the precise consistency strength 
of $\Det(\mbf\Sigma^{1}_{\alpha})$ for a~given~$\alpha$? 
for all~$\alpha$, i.e., of $\sigma$-$\PD$? 
is it weaker than $\Det$(Hyp)? than $\AD$?

As well-known [references??], 
\begin{enumerate}
\item[(i)]
$n$~Woodin cardinals and a~measurable above 
implies $\Det(\mbf\Sigma^{1}_{n+1})$,
\item[(ii)]
$\omega$~Woodin cardinals implies~$\PD$ 
(this follows from~(i)),
\item[(iii)]
$\omega$~Woodin cardinals and a~measurable above 
implies $\AD^{L(\mathbb R)}$ (and hence, by the 
lemma above, $\Det(\mbf\Sigma^{1}_{\alpha})$ for 
all $\alpha<\omega_1$ [thus we do not need, e.g., 
$\alpha$~Woodins and a~measurable above; it seems, 
$\alpha$~Woodins give something for longer games
-- double check!!].
\end{enumerate}
\end{q}

\begin{tm}
Suppose $\DC$ and assume $\AD^{L(\mathbb R)}$ 
[in fact, $\sigma$-$\PD$ suffices!!!]. Then for any 
Tychonoff space~$X$ and all $\alpha\in\omega_1\setminus2$, 
the classes $\mbf\Sigma^{1}_{2\alpha}(\mathscr Z,X)$ 
and $\mbf\Pi^{1}_{2\alpha+1}(\mathscr Z,X)$ 
have the reduction property, while 
$\mbf\Sigma^{1}_{2\alpha+1}(\mathscr Z,X)$ 
and $\mbf\Pi^{1}_{2\alpha}(\mathscr Z,X)$ 
have the separation property.
\end{tm}

\begin{proof}
This follows from [the lemma above, i.e.,
the First Periodicity Theorem for $\sigma$-projective
classes in~$\omega^\omega$] and Theorem~\ref{t: red tych}.
\end{proof}

\begin{q}
Does not the pre-well-ordering property itself 
hold for these classes? 
What about stronger properties as scales? 

(Note that $\sigma$-$\PD$ implies the scale property 
for $\sigma$-projective sets in~$\omega^\omega$, 
see \cite{Di Prisco et al 1982}, Proposition~8, 
or \cite{Kechris}, Exercise~39.18.)
\end{q}


\begin{q}
Check whether the following is true: 
if $A\le_{\mathrm W}B$ (the Wadge ordering),
then the Hausdorff operation with the base~$B$ is 
stronger than one with the base~$A$. Consequently, 
under $\AD^{L(\mathbb R)}$, the Hausdorff operations 
with bases in~$L(\mathbb R)$ are well-ordered.

By \cite{Kantorovich Livenson}, Theorem~8.1, 
$\varPhi_S\le\varPhi_T$ iff $S\subseteq\widetilde{T}$, 
where 
$
\widetilde{T}=
\{g\in\omega^\omega:(\exists f\in T)\,\ran g=\ran f\}
$ 
(note that, as $\ran g=\ran f$ implies
$\bigcap_nA_{f\uhr n}=\bigcap_nA_{g\uhr n}$, 
we clearly have $\varPhi_T=\varPhi_{\widetilde{T}}$). 

\end{q}

\end{footnotesize}

\newpage

\begin{footnotesize}

\section*{$\varPhi$-hierarchy: 
pre-well-ordering, reduction, separation}

Here we consider the hierarchy constructed by a~given 
Hausdorff operation~$\varPhi$ [of countable arity for 
the moment] and its dual operation~$\varPsi^*$ as follows: 
\begin{align*}
\varPhi_{0}(\mathscr S,X)
&=
\mathscr S(X),
\\
\varPhi_{\alpha}(\mathscr S,X) 
&=
\bigl\{\varPhi(A_n)_{n<\omega}:
(A_n)_{n<\omega}\in 
\bigl(\,{\textstyle\bigcup_{\beta<\alpha}}
\varPhi^{*}_{\beta}(\mathscr S,X)\bigr)^\omega
\bigr\}.
\end{align*}
If $\varPhi$~is the operation of countable union, 
we get the Borel hierarchy; if $\varPhi$~is the 
$A$-operation, we get the hierarchy of $C$-sets 
(see, e.g., \cite{Kanovei 1988},~\S2). 
In general, we can call it 
the $C$-{\it like hierarchy given by}~$\varPhi$.

\begin{lm}
Let $\varPhi$~be a~Hausdorff operation with a~base 
in $L(\mathbb R)$, and let $\mathscr S\subseteq\omega^\omega$ 
be a~family in $L(\mathbb R)$.
Then for all $\alpha\in\omega_1$, 
the classes $\varPhi_{\alpha}(\mathscr S)$ 
and $\varPhi^{*}_{\alpha}(\mathscr S)$ 
are in $L(\mathbb R)$.
\end{lm}

\begin{proof} 
Recall that the operation~$\varPhi$ with a~base 
$S\subseteq\omega^\omega$ applied to sets in 
$\mathscr S(X)$ gives the same result that the 
projection of sets in $\mathscr S(X\times S)$ onto~$X$. 
Noting now that, whenever $A$ and~$B$ are 
in $L(\mathbb R)$ then so is $A\times B$ as well as 
$\pr_A(C)$ for all $C\subseteq A$ in $L(\mathbb R)$, 
and also that whenever $\varPhi$~has a~base in $L(\mathbb R)$ 
then so does~$\varPhi^{*}$, we prove the claim 
by induction on~$\alpha$.

[We may also prove by induction on~$\alpha$ that 
$\varPhi_\alpha$~is defined by a~certain single Hausdorff 
operation with a~base in $L(\mathbb R)$.]
\end{proof}

[The following can be considered as the First 
Periodicity Theorem for the $\varPhi$-hierarchy:]

\begin{lm}
Assume $\AD^{L(\mathbb R)}$. 
Let $\varPhi$~be a~Hausdorff operation with a~base in 
$L(\mathbb R)$, and let $\mathscr S$ be a~class such that
$\mathscr S(\omega^\omega)\subseteq\mathscr P(\omega^\omega)$ 
is in $L(\mathbb R)$ and has the pre-well-ordering property. 
For all $\alpha\in\omega_1\setminus1$, 
the classes $\varPhi_{2\alpha}(\mathscr S,\omega^\omega)$ 
and $\varPhi^{*}_{2\alpha+1}(\mathscr S,\omega^\omega)$ 
have the pre-well-ordering property. 
Consequently, these classes have the reduction property, 
while the dual classes 
$\varPhi_{2\alpha+1}(\mathscr S,\omega^\omega)$ 
and $\varPhi^{*}_{2\alpha}(\mathscr S,\omega^\omega)$ 
have the separation property.
\end{lm}

\begin{proof} 
Use [the previous lemma] and two facts used in 
[the theorem on $\sigma$-projectives.] 
[double check!!!]
\end{proof}


\begin{tm}
Assume $\DC$ and $\AD^{L(\mathbb R)}$. 
Let $X$~be a~Tychonoff space.
Then for all $\alpha\in\omega_1\setminus2$, 
the classes $\varPhi_{2\alpha}(\mathscr Z,X)$ 
and $\varPhi^{*}_{2\alpha+1}(\mathscr Z,X)$ 
have the reduction property, while the dual 
classes $\varPhi_{2\alpha+1}(\mathscr Z,X)$ 
and $\varPhi^{*}_{2\alpha}(\mathscr Z,X)$ 
have the separation property.
\end{tm}

\begin{proof}
This follows from [the lemma above, i.e., the First 
Periodicity Theorem for the $\varPhi$-hierarchy
in~$\omega^\omega$] and Theorem~\ref{t: red tych}.
\end{proof}

\begin{q}
Does not the pre-well-ordering property itself 
(or even the scale property etc.)~hold for 
these classes $\varPhi_{\alpha}(\mathscr Z,X)$? 
\end{q}

\begin{q}
Similarly for the $R$-like hierarchy.
\end{q}

\begin{q}
Write on countable reduction/separation properties 
in Tychonoff (or more general) spaces.
\end{q}

\end{footnotesize}

\newpage

}

\subsection*{Acknowledgements.}
I am indebted to Sergei V.~Medvedev who carefully 
read an earlier version of this note and pointed 
out some inaccuracies in it. I~am also grateful 
to Olga V.~Sipacheva and Vladimir G.~Kanovei 
for helpful discussions and remarks.


\begin{footnotesize} 
\noindent
{\sc
The Russian Academy of Sciences, 
Institute for Information Transmission Problems, 
Bolshoy Karetny lane~19, building~1, 
Moscow 127051 Russia
\/}
\\
{\it E-mail address:\/} 
d.i.saveliev@iitp.ru, 
d.i.saveliev@gmail.com 
\end{footnotesize}


\begin{thebibliography}{111}

\hide{
\bibitem{Chernavski} 
A.\,V.~Chernavski,
{\it A~remark to Schneider's theorem on 
the existence of an $A$-set which is not a~$B$-set 
in perfectly normal bicompacts\/}.
Vestnik MSU, Math.~Mech., 2 (1962), 20.
}


\bibitem{Choban 1989} 
M.\,M.~Choban,
{\it Descriptive set theory and topology}.
Itogi Nauki i Tekhniki. 
Ser. Sovrem. Probl. Mat. Fund. Napr., 51, 
VINITI, Moscow, 1989, 173--237 
(in Russian).

\hide{
\bibitem{Choban 2005} 
M.\,M.~Choban,
{\it On some problems of descriptive 
set theory in topological spaces}.
Russian Math. Surveys,  60:4 (2005), 699--719.
}

\bibitem{Dasgupta} 
A.~Dasgupta, 
{\it Boolean operations, Borel sets, and 
Hausdorff's question}. 
Journal of Symbolic Logic, 61 (1996), 1287--1304.

\bibitem{Di Prisco et al 1982}
C.\,A.~Di Prisco, W.~Marek,
{\it On some $\sigma$-algebras containing 
the projective sets}.
Zeitschrift f{\"u}r mathem. Logic und 
Grundlagen der Mathem., 
28 (1982), 526--638.

\bibitem{Di Prisco et al 1987}
C.\,A.~Di Prisco, H.~Llopis,
{\it On some extensions of the projective hierarchy}.
Annals of Pure and Applied Logic 36 (1987), 105--113.

\bibitem{Dougherty} 
R.~Dougherty, 
{\it Sequential discreteness and 
clopen-$I$-Boolean classes}, 
Journal of Symbolic Logic, 52:1 (1987), 232--242.

\bibitem{Engelking} 
R.~Engelking.
{\it General topology\/}.
PWN, Warszawa, 1977.

\hide{
\bibitem{Frolik} 
Z.~Frol{\'i}k.
{\it On the topological product of paracompact spaces\/}.
Bull. Acad. Polon., 8 (1960), 747-750.
}

\hide{
\bibitem{Gao} 
S.~Gao. 
{\it Invariant descriptive set theory.\/}
CRC Press, 2008.
}

\bibitem{Kanamori} 
A.~Kanamori. 
{\it The higher infinite\/}.
Springer, 1994.

\bibitem{Kanovei 1988}
V.\,G.~Kanovei,
{\it Kolmogorov's ideas in the theory 
of operations on sets}. 
Russian Math. Surveys, 43:6 (1988), 111--155. 

\hide{
\bibitem{Kanovei Lyubetsky 2019}
V.\,G.~Kanovei, V.\,A.~Lyubetsky,
{\it Models of set theory in which 
separation theorem fails}. 
2019, arXiv:1905.11241.
}

\bibitem{Kantorovich Livenson}
L.\,V.~Kantorovich, E.\,M.~Livenson,
{\it Memoir on the analytical operations 
and projective sets}. 
Fundam. Mathem., 
I: 18 (1932), 214--279, II:  20 (1933), 54--97. 

\bibitem{Kechris} 
A.\,S.~Kechris.
{\it Classical descriptive set theory.\/}
Springer, 1994.

\bibitem{Kechris et al 2008} 
A.\,S.~Kechris, B.~L{\"o}we, J.\,R.~Steel (eds.).
{\it Games, scales, and Suslin cardinals: 
The Cabal seminar, vol.~I.\/}
Lecture Notes in Logic~31, 
Cambridge Univ.~Press, 2008. 

\bibitem{Kechris et al 2011} 
A.\,S.~Kechris, B.~L{\"o}we, J.\,R.~Steel (eds.).
{\it Wadge degrees and projective ordinals: 
The Cabal seminar, vol.~II.\/}
Lecture Notes in Logic~37,  
Cambridge Univ.~Press, 2011. 

\bibitem{Miller} 
A.\,W.~Miller, 
{\it Projective subsets of separable metric spaces}.
Annals of Pure and Applied Logic, 50 (1990), 53--69.

\bibitem{Moschovakis} 
Y.\,N.~Moschovakis. 
{\it Descriptive set theory.\/}
Second edition, Mathematical surveys 
and monographs~155, 2009.

\hide{
\bibitem{Ponomarev} 
V.\,I.~Ponomarev. 
{\it On Borel sets in perfectly normal compact spaces.\/}
Doklady AN, 170:3 (1966), 520--523.
}

\hide{
\bibitem{Przymusinsky} 
T.~Przymusi{\'n}sky.
{\it Perfectly normal compact spaces are 
continuous images of $\beta N\setminus N$\/}. 
Proc. AMS, 96:3 (1982), 541--544.
}

\hide{
\bibitem{Steen Seebach} 
L.\,A.~Steen, J.\,A.~Seebach (Jr.),  
{\it Counterexamples in Topology}. 
Dover Publications reprint of 1978~ed., 
Berlin, New York, Springer-Verlag, 1995.
}

\hide{
\bibitem{Ulam} 
S.\,M.~Ulam, 
{\it Problems in modern mathematics}.
Wiley, N.\,Y., 1964.
}

\end{thebibliography}
\end{document}